\newif\ifIJM
\IJMfalse

\ifIJM
\documentclass{ijmart}
\else
\documentclass[a4paper,parskip,abstract]{scrartcl}
\fi

\usepackage[utf8]{inputenc}
\usepackage{amsmath}
\usepackage{amsfonts}
\usepackage{mathtools}
\usepackage{faktor}
\usepackage{nicefrac}
\usepackage{enumerate}
\usepackage{dsfont}
\usepackage{aliascnt}
\usepackage[pdfauthor={Joachim Breitner},pdftitle={Loop subgroups of Fr and the image of their stabilizer subgroups in GLr(Z)},pdfborder={0 0 .2}]{hyperref}
\ifIJM
\else
\usepackage[amsmath,thmmarks]{ntheorem}
\fi
\usepackage{tikz}
\usetikzlibrary{automata}
\usetikzlibrary{decorations}
\usepackage{microtype}

\usepackage[numbers,square]{natbib}

\ifIJM\else
\theoremnumbering{arabic}
\theoremstyle{plain}
\theorembodyfont{\itshape}
\theoremheaderfont{\normalfont\bfseries}
\theoremseparator{}
\fi

\newtheorem{theorem}{Theorem}
\newtheorem{proposition}{Proposition}

\newaliascnt{lemma}{proposition}
\newtheorem{lemma}[lemma]{Lemma}
\aliascntresetthe{lemma}

\newaliascnt{corollary}{proposition}
\newtheorem{corollary}[corollary]{Corollary}
\aliascntresetthe{corollary}

\ifIJM\else
\theorembodyfont{\upshape}
\fi
\newaliascnt{example}{proposition}
\newtheorem{example}[example]{Example}
\aliascntresetthe{example}

\newaliascnt{remark}{proposition}
\newtheorem{remark}[remark]{Remark}
\aliascntresetthe{remark}

\newaliascnt{definition}{proposition}
\newtheorem{definition}[definition]{Definition}
\aliascntresetthe{definition}

\ifIJM\else
\theoremstyle{nonumberplain}
\theoremheaderfont{\scshape}
\theorembodyfont{\normalfont}
\theoremsymbol{\ensuremath{_\blacksquare}}
\newtheorem{proof}{Proof}
\fi

\ifIJM\else
\makeatletter
\let\my@@font@warning\@font@warning
\let\@font@warning\@font@info
\makeatother
\usepackage{fourier}
\makeatletter
\let\@font@warning\my@@font@warning
\makeatother 
\fi

\DeclareMathOperator{\Aut}{Aut}
\DeclareMathOperator{\GL}{GL}
\DeclareMathOperator{\SL}{SL}
\DeclareMathOperator{\Stab}{Stab}
\DeclareMathOperator{\Id}{Id}

\DeclareMathOperator{\NT}{NT}
\DeclareMathOperator{\Diag}{Diag}
\DeclareMathOperator{\Sym}{Sym}
\newcommand{\da}{\coloneqq}

\newcommand{\Z}{\mathds Z}
\newcommand{\N}{\mathds N}
\newcommand{\ev}{\mathds 1}

\author{Joachim Breitner}
\ifIJM
\address{Karlsruhe Institute of Technology, Germany}
\email{mail@joachim-breitner.de}
\title[Loop subgroups of $F_r$]{Loop subgroups of $F_r$ and the image of their stabilizer subgroups in $\GL_r(\Z)$}
\else
\title{Loop subgroups of $F_r$ and the image of their stabilizer subgroups in $\GL_r(\Z)$}
\fi
\begin{document}
\ifIJM\else
\maketitle
\fi

\begin{abstract}
\noindent For a representative class of subgroups of $F_r$, the image of their stabilizer subgroup under the action of $\Aut(F_r)$ in $\GL_r(\Z)$ is calculated.
\end{abstract}

\ifIJM
\maketitle
\fi


\section{Preface}

Let $F_r$ be the free group of rank $r$. $\Aut(F_r)$ naturally acts on the set of subgroups of fixed index $n$ and gives rise to the stabilizer $\Stab_{\Aut(F_r)}(U)$ of such a subgroup $U$.
Here, we study its image in $\GL_r(\Z)$ under the map $B:\Aut(F_r) \to \GL_r(\Z)$ induced by the abelianization of $F_r$ and show that for a “general subgroup” $U$ and $r$ large enough, the image $B(\GL_r(\Z))$ is a congruence subgroup of level two.

To this end we consider a certain class of subgroups of $F_r$, which we dub \emph{loop subgroups}, due to the appearance of their coset graphs (cf.\ \autoref{loopsubgroupdef}) with regard to a suitable set of generators of $F_r$, and obtain as the main result the following theorem:

{
\noindent\textbf{\autoref{sharpbound}} \itshape
For a loop subgroup  $U\le F_r$, $r\ge 3$, with at most $r-2$ looplets, we have
\[
B(\Stab_{\Aut(F_r)}(U)) = \{M \in \GL_r(\Z) \mid  v \cdot M \equiv v \pmod 2\}
\]
for a particular row vector $v\in\Z^r$.
}


Jan-Christoph Schlage-Puchta pointed out to us that for large $r$ in relation to the index $n=[F_r:U]$, there are only two large $\Aut(F_r)$-orbits, both of which consist of loop subgroups. This follows from a result by John D. Dixon \citep[Theorem 1]{Dixon}, whereby the probability that two random permutations generate $S_n$ or $A_n$ is high, and Robert Gilman \citep[Proof of Theorem 3]{Gilman}, whereby all generating $G$-vectors of length $r$ lie in the same orbit under the action of $\Aut(F_r)$, if $r$ is large enough in relation to $|G|$. Therefore \autoref{sharpbound} holds for the “general” subgroups.

This articles was motivated by the study of imprimitive translation surfaces, especially origamis \cite{Sch04}, where the map $B:\Aut(F_r)\to \GL_r(\Z)$ is used to obtain Veech groups of translation surfaces. An origami defines a special type of translation surface, also called square-tiled surface, coming from a covering of the torus. Each origami can be associated with a subgroup $U\le F_2$ of finite index. The image of the stabilizer group of $U$ in $\Aut(F_2)$ under $B$ intersected with $\SL_2(\Z)$ gives the Veech group of the origami.
In the case of origamis, an interesting question is which subgroups of $\SL_2(\Z)$ occur as Veech groups. There is a positive answer for many congruence subgroups (see \cite{Sch05}), and for all subgroups of the principal congruence group of level 2 that contain $-I$ (see \citep[Theorem 1.2]{ellenberg-2009}). Further results about Veech groups of coverings of $n$-gons are found in \cite{Fin2010}.
The construction of a Veech group can be generalized to subgroups of $F_r$ for a general $r\in\N$ instead of origamis. For rank 3 or higher,  there is a reason to hope that these groups will be easier to understand, as every normal subgroup of $\SL_r(\Z)$ besides $\{I\}$ and $\{I,-I\}$ is a congruence group and can be thought of as a subgroup of $\SL_r(\Z/l\Z)$, where $l$ is the congruence level of the subgroup, see e.g. \cite{Sury}.
The loop subgroups studied here can be considered as generalizations of L-origamis, which are studied in detail in \cite{HubLe}. We find that for loop subgroups, their analogs to Veech groups show very different behavior.

\ifthenelse{\isundefined{IJM}}
{
The following section begins with the definition of a loop subgroup. We obtain some lemmata about permutations in \autoref{someperm} and about generators of certain matrix groups in section \ref{linergroupsec}~and~\ref{stabsubgroups}. In \autoref{preimages}, preimages in $\Aut(F_r)$ to these generators are constructed and used in \autoref{mixedcase} in \autoref{lowerboundsec} to obtain a lower bound of the image under $B$ in $\GL_r(\Z)$ of the stabilizer subgroup of a loop subgroup. The lower bound is also the upper bound, as shown in \autoref{generalupperbound} in \autoref{upperboundsec}. This leads to the the main result in \autoref{sharpbound}. Similar reasoning works for the subgroups excluded, which are handled in \autoref{excludedcasesec}.

}{
}

This paper
was written under the guidance of Gabriela Schmithüsen, who was generous with time, advice, regular proof-reading, ideas and inspiration. Fruitful discussions on the topic arouse with Myriam Finster. The work was partially supported by the Landesstiftung Baden-Württemberg 
within the project “With origamis to Teichmüller curves in moduli space”.

\section{Loop subgroups}
\label{loopsubgroupdefsec}


\begin{definition}
\label{loopsubgroupdef}
A subgroup $U\le F_r$ is called a $s_1/\ldots/s_r$ \emph{loop subgroup}, $s_i\in\N$, if there is a 
a basis $\{g_1,\ldots,g_r\}$ of $F_r$ such that 
\begin{itemize}
\item the action of $g_i$ on the left cosets of $U$ is a permutation consisting of one cycle of length $s_i$ and no other non-trivial cycles,
\item each such non-trivial cycle includes the coset $U$ and
\item these non-trivial cycles are otherwise distinct.
\end{itemize}
\end{definition}

The sequence of nodes $U,g_i^1U,\ldots, g_i^{s_i-1}U, U$ is called a \emph{loop}, where $s_i$ is the the \emph{length} of the loop. A loop is called \emph{odd} (resp. \emph{even}) if its length is odd (resp. even). A loop of length 1 is called \emph{looplet}\footnote{The German language allows to build diminutive forms of almost all nouns by appending the suffix \emph{-chen}. I take the liberty to do the same in English, as it makes the text easier and more pleasant to read than if I had named them \emph{small loops}.}.

This definition is invariant under the natural action of $\Aut(F_r)$ on subgroups of $F_r$. In the remainder of this article, $\{g_1,\ldots,g_r\}$ is a basis of $F_r$ as in the above definition. If we fix a set of generators, the lengths of the loops $s_1/\ldots/s_r$ fully determine the subgroup and allow us to speak of \emph{the} loop subgroup.

The reason for our nomenclature becomes evident if we draw the coset graph of a loop subgroup:

\begin{example}
\label{loopex}The coset graph of the $3/3/1$ loop subgroup of $F_3$, a subgroup with only odd loops and one looplet, is shown in \autoref{loopexfigure}.
\end{example}
\begin{figure}[hbt]
\begin{center}
\begin{tikzpicture}[auto]
\draw (0,0) node[state,accepting,draw,circle] (1) {$U$};
\begin{scope}[xshift=2cm]
\path (1) --
      (canvas polar cs:radius=2cm, angle=60) node[state,draw,circle] (2) {$y^1U$} --
      (canvas polar cs:radius=2cm, angle=-60) node[state,draw,circle] (3) {$y^2U$};
\draw[->] (canvas polar cs:radius=2cm, angle=160) arc(160:120:2cm) node[above] {$y$} arc(120:80:2cm);
\draw[->] (canvas polar cs:radius=2cm, angle=40) arc(40:0:2cm) node[right] {$y$} arc(0:-40:2cm);
\draw[->] (canvas polar cs:radius=2cm, angle=-80) arc(-80:-120:2cm) node[below] {$y$} arc(-120:-160:2cm);
\draw[->] (2) edge [in=45,out=75,loop] node {$z,x$} (2);
\draw[->] (3) edge [in=-75,out=-45,loop] node {$z,x$} (3);
\end{scope}
\begin{scope}[xshift=-2cm]
\path (1) --
      (canvas polar cs:radius=2cm, angle=120) node[draw,circle] (2) {$x^1U$} --
      (canvas polar cs:radius=2cm, angle=-120) node[draw,circle] (3) {$x^2U$};
\draw[->] (canvas polar cs:radius=2cm, angle=20) arc(20:60:2cm) node[above] {$x$} arc(60:100:2cm);
\draw[->] (canvas polar cs:radius=2cm, angle=140) arc(140:180:2cm) node[left] {$x$} arc(180:220:2cm);
\draw[->] (canvas polar cs:radius=2cm, angle=-100) arc(-100:-60:2cm) node[below] {$x$} arc(-60:-20:2cm);
\draw[->] (2) edge [in=105,out=135,loop] node {$z,y$} (2);
\draw[->] (3) edge [in=225,out=255,loop] node {$z,y$} (3);
\end{scope}
\draw[->] (1) edge [loop left] node {$z$} (1);
\end{tikzpicture}
\end{center}
\caption{The left coset graph of the 3/3/1 loop subgroup of $F_3$.}
\label{loopexfigure}
\end{figure}
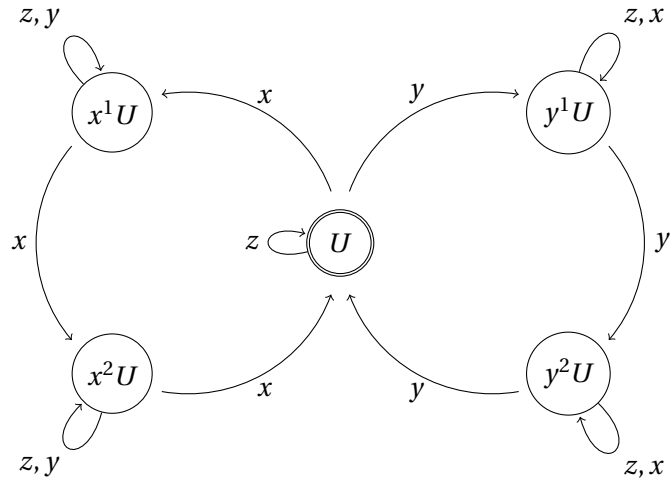

The following set of words in $F_r$ is a basis of the $s_1/\ldots/s_r$ loop subgroup of $F_r$ with respect to the basis $\{g_1,\ldots,g_r\}$
\begin{align*}
\{ g_i^{s_i}, \ i=1,\ldots,r \} \cup\ \{ g_i^{-k}g_j^{} g_i^k, \ i,j=1,\ldots,r,\, i\ne j,\, k=1,\ldots, s_i-1\}. 
\end{align*}

Before we start investigating loop subgroups, some preparational definitions and calculations are due. 

\section{Permutations of cosets}
\label{someperm}

For a subgroup $U$ of $F_r$ we formalize the coset action of $F_r$ by the homomorphism $\pi\colon  F_r  \to \Sym(N)$, $\pi(w)(vU) \da wvU$, where $N$ denotes the set of left cosets of $U$ and $\Sym(N)$ is the symmetric group thereon.

\begin{remark}
\label{piprops}
Recall the following useful properties of $\pi$ for $w\in F_r$:
\begin{enumerate}
\item $\pi(w)(U) = U \iff w\in U$.
\item $\pi(w) = \Id \iff w \in \NT(U) \da \bigcap_{v\in F_r} v^{-1}U v$.
\item For $\gamma\in\Stab_{\Aut(F_r)}(U)$,
$\pi(\gamma(w))$ is conjugate to $\pi(w)$. 
\end{enumerate}
\end{remark}

Part 3 of the previous lemma can be formulated more precisely as $\pi(\gamma(w)) = \pi(\gamma) \circ \pi(w) \circ \pi(\gamma)^{-1}$, where the permutation $\pi(\gamma)$ is defined by abuse of notation as $\pi(\gamma)(vU) \da \gamma(v)U$ for $v\in F_r$. This is well defined, as stabilizers map left cosets to left cosets. 

We will repeatedly construct automorphisms of the following form:
\begin{remark}
\label{autconstr}A map $\gamma : F_r \to F_r$ defined by
\[
\gamma(g_j) = 
\begin{cases}
w g_i, &\text{if } j = i \\
\phantom w g_j, &\text{if } j \ne i
\end{cases}
\]
is an automorphism if the generator $g_i$ does not occur in $w$, i.e.
\[
w\in \langle g_1,\ldots,g_{i-1},g_{i+1},\ldots,g_r\rangle.
\]
It stabilizes the subgroup $U$ if $w\in\NT(U)$, as then $\pi(w)=\Id$ holds, implying $\pi(\gamma(v)) = \pi(v)$ and thus $\gamma(v)\in U \iff v\in U$.
\end{remark}

The following fact about permutations will be very useful for our later calculations:
\ifIJM
\begin{remark}
\else
\begin{lemma}
\fi
\label{altperm}
Let $\sigma, \omega \in S_n$ be permutations of the form $\sigma = (1,2,\ldots,m)$ and $\omega=(1,m+1,\ldots,n)$ with $1<m<n$. Then all even permutations are in the commutator subgroup of the group generated by $\sigma$ and $\omega$, i.e.\ they can be written as a product of commutations of $\sigma$ and $\omega$: 
\[
A_n \le [\langle \sigma, \omega \rangle, \langle \sigma, \omega \rangle]
\]
where $A_n$ is the alternating group of degree $n$ and $[G,G]$ denotes the commutator subgroup of $G$.
\ifIJM
\end{remark}
\else
\end{lemma}
\fi

\ifIJM\else
\begin{proof}
The alternating group $A_n$, $n\ge 3$, is generated by all three-cycles in $S_n$. We first generate all three-cycles which do not fix the 1. These are cycles of the form $(1,i,j)$. There are four cases:
\begin{itemize}
\item $1<j\le m$ and $m < i \le n$. In this case, we take the inverse and enter the next case.
\item $1<i\le m$ and $m < j \le n$. We write the cycle using commutators of $\omega$ and $\sigma$:
\[
(1,i,j) = \sigma^{i-1} \circ \omega^{j-m} \circ \sigma^{-(i-1)} \circ \omega^{-(j-m)}
\]
\item $1<i\le m$ and $1< j \le m$. We can reduce this case to the previous by writing:
\[
(1,i,j) = (1,j,m+1)^{-1}\circ (1,i,m+1)
\]
\item $m<i\le n$ and $m< j \le n$. This case works analogously to the previous case:
\[
(1,i,j) = (1,2,j)\circ (1,2,i)^{-1}
\]
\end{itemize}

Observe that $\omega \circ \sigma = (1,2,\ldots,n)$.
If we now have a general three-cycle $(k,i,j)$, it is conjugate to a three-cycle that moves the 1:
\[
(k,i,j) = (\omega\circ\sigma)^{k-1} \circ (1,i-(k-1),j-(k-1)) \circ (\omega\circ\sigma)^{-(k-1)}
\]
So we can write any even permutation in $S_n$ as a product of $\omega$’s and $\sigma$’s, such that for either of the two generators, the sum of its occurrences, counting negative powers negatively, is zero.
\end{proof}
\fi

\section{The linear group and the principal congruence subgroup}
\label{linergroupsec}

The elementary matrix $X_{ij} \in \GL_r(\Z)$, $i\ne j$ is the matrix with ones on the diagonal, one additional one in the $i$-th row and $j$-th column and zeroes everywhere else. The elementary matrices generate $\SL_r(\Z)$. Another generating set is 
\[
\{X_{ik}, k\ne i\} \cup \{ X_{ji}, j\ne i\}
\]
for a fixed $i$. This follows from the relations $[X_{ji},X_{ik}] = X_{jk}$ for $j\ne k$ \citep[Theorem 4-3.2]{Sury}.

The matrix $T_i \da \Diag(1,\ldots,1,-1,1,\ldots,1)$ is defined as the identity matrix with the exception of one $-1$ on the diagonal in the $i$-th row. Together with $T_1$ either of the generating sets above generate the whole group $\GL_r(\Z)$.

\begin{remark}
\label{gamma2gens}
For the principal congruence subgroup of level two
\[
\Gamma_2 = \{A \in \GL_r(\Z) \mid A\equiv \Id \pmod 2\},
\]
a similar generating set 
\[
\{X_{ij}^2,\, i\ne j\}  \cup \{T_i, i=1,\ldots,r\}
\]
consisting of squares of elementary matrices and diagonal matrices which are the identity matrix with the exception of one $-1$ on the diagonal can be given. This can be proven by transforming a matrix $M\in\Gamma_2$ to the identity matrix using the row- and column-operations described by the alleged generators.
\end{remark}

\begin{corollary}
\label{gamma2gensalt}
The set
\[
\{X_{ik}, k\ne i\} \cup \{ X_{ji}^2, j\ne i\} \cup \{T_j, j=1,\ldots,r\}
\]
for some fixed $i\in\{1,\ldots,r\}$ generates a superset of $\Gamma_2$.
\end{corollary}

This follows from the relation $X_{jk}^2 = [X_{ji}^2,X_{ik}]$ for distinct $i,j,k$.

\section{Stabilizer subgroups in \texorpdfstring{$\GL_r(\Z/2\Z)$}{GLr(Z/2Z)}}
\label{stabsubgroups}

When we inspect the loop subgroup with even loops, we will come across the group of integral matrices with odd column sums. It contains the principal congruence subgroup of level 2, hence it is completely determined by its image in $\GL_r(\Z/2\Z)$:
\[
S(\ev) \da \{ M\in \GL_r(\Z/2\Z) \mid \ev \cdot M = \ev \}
\]
where $\ev = (1,\ldots,1) \in (\Z/2\Z)^r$ is the row vector, all of whose entries are one. More generally, we will be interested in the stabilizer subgroup of a row vector $v=(v_1,\ldots,v_r)\in (\Z/2\Z)^r$ under the action of right multiplication:
\[
S(v) \da \{ M\in \GL_r(\Z/2\Z) \mid v \cdot M = v \}
\]

\begin{lemma}
\label{stabgens}The group $S(\ev)$ is generated by “double elementary matrices” of the form $X_{ik} X_{jk}$, $i,j,k\in\{1,\ldots,r\}$ pairwise distinct, having ones on the diagonal and in two additional positions in the same column, and zeros everywhere else.

More general, $S(v)$ is generated by the elementary matrices $X_{ij}$ for $i\ne j$ and $v_i=0$ and the double elementary matrices $X_{ik}X_{jk}$ for $v_i=v_j=1$, $i,j,k\in\{1,\ldots,r\}$ pairwise distinct.
\end{lemma}

Note that the case $v=0$ is covered by the lemma, as $S(0) = \GL_r(\Z/2\Z)$ is generated by all the elementary matrices $X_{ij}$, $i\ne j$. If the vector $v$ is a unit vector, i.e.\ precisely one entry is $1$, the generating set does indeed not contain any double elementary matrices.

\begin{proof}
Direct computation shows that the alleged generators are in $S(v)$.

Recall that the multiplication from the left of an elementary matrix $X_{ij}$ has the effect of adding the $j$-th row on the $i$-th row and the multiplication from the left of a double elementary matrix $X_{ik}X_{jk}$ has the effect of adding the $k$-th row simultaneously on the $i$-th and the $j$-th row. 

We will now transform a matrix in $S(v)$ into the identity matrix by multiplying the alleged generators from the left. To do so, we suppose that the first $i-1$ columns are already that of the identity matrix and treat column $i\in\{1,\ldots,r\}$ as follows:

\begin{enumerate}
\item Ensure that there is a 1 on the diagonal. If there is none, there still must be at least one 1 in the column, otherwise the matrix would not be regular. There even must be a 1 in a row $j$ with $j > i$, as otherwise the current column would be a linear combination of the columns to the left of it, which are, due to our transformation, the unit vectors $e_1,\ldots,e_{i-1}$.

Now add this row $j$ on the $i$-th row (multiplication with $X_{ij}$).  If $v_i=1$, then, to be allowed to do that, also add the $j$-th row on any other row $k$ with $v_k=1$ (multiplication with $X_{ij}X_{kj}$). This step does not alter the previous columns, as the row $j$ has zeros there.

If there is no such other row $k$, that is, if there are exactly two ones in the vector $v$, namely $v_i=v_j=1$, and if the current column has only one 1 in the $j$-th row, some additional shuffling is necessary. Let $k\ne i,j$ be any other row. This implies $v_k=0$. The matrix
\[
(X_{ik}X_{jk}\cdot X_{ki}\cdot X_{kj})^2
\]
is actually the permutation matrix that swaps the $i$-th and $j$-th row, and is here written in terms of the given generators. Multiplying this matrix from the left, we move the 1 to the right spot, while again not altering the previous columns.

\item Eliminate all ones that are not on the diagonal. Ones in rows $j$ with $v_j=0$ are eliminated directly with $X_{ji}$. Ones in rows $j$ and $k$ with $v_j=v_k=1$ can only be eliminated pairwise with $X_{ji}X_{ki}$. But in any case there is an even number of them: Either $v_i=0$. Then we know from the equation $v\cdot M = v$ that there is an even number of ones to be eliminated. Or $v_i=1$, then there is an odd number of ones in this column. $v_i = 1$ means that we count the one on the diagonal, which we want to retain, leaving us with an even number of ones to eliminate.
\end{enumerate}
So any matrix in $S(v)$ can be transformed into the identity matrix using the given generators, thus they indeed generate all of $S(v)$.
\end{proof}

%
\section{Preimages}
\label{preimages}

After these calculations we will start investigating $B(\Stab_{\Aut(F_r)}(U))$ for a loop subgroup $U$. Recall that the map $B$ is the map naturally induced by the abelianization $F_r \to \Z^r$, which can be given explicitly with regard to a basis $\{g_1,\ldots,g_r\}$ of $F_r$:

\begin{definition}
\label{defB}
The map $B : \Aut(F_r) \to \GL_r(\Z)$ is defined by 
\[
\gamma \mapsto \big(\#_{g_i}\,\gamma(g_j)\big)_{i,j=1\ldots r}\,,
\]
where  $\#_{g_i} : F_r \to \Z$ is the map that counts the $i$-th generator, sending $g_i \mapsto 1$ and $g_j \mapsto 0$ for $j\ne i$.
\end{definition}

To find a lower bound of the group $B(\Stab_{\Aut(F_r)}(U))$, the following three lemmata identify sufficient conditions for elementary matrices, squared elementary matrices or double elementary matrices to be in this group.

\begin{lemma}
\label{oddcasepreimage}
Let $U\le F_r$, $r\ge 3$, be the $s_1/\ldots/s_r$-loop subgroup, $i,j\in\{1,\ldots,r\}$ with $i\ne j$ and $s_i$ odd. If $s_i =1$ or there exists $k\in\{1,\ldots,r\}$, $k\ne i,j$, with $s_k>1$, then 
\[
X_{ij} \in B(\Stab_{\Aut(F_r)}(U)).
\]
\end{lemma}

\ifIJM\else
\begin{example}
For the 3/3/1 loop subgroup seen in \autoref{loopex}, the following preimages to the elementary matrices $X_{31}$, $X_{32}$, $X_{13}$ and $X_{23}$ are constructed:
\begin{align*}
\gamma_{31}(x,y,z) &= (z\cdot x,y,z) \\
\gamma_{32}(x,y,z) &= (x,z\cdot y,z) \\
\gamma_{13}(x,y,z) &= (x,y, yxy^{-1}xyx^{-2}y^{-1} \cdot x \cdot z) \\
\gamma_{23}(x,y,z) &= (x,y, \underbrace{xyx^{-1}yxy^{-2}x^{-1}}_{\in [\langle x,y\rangle,\langle x,y\rangle]} \cdot y \cdot z).
\end{align*}
\end{example}
\fi

\ifIJM
\begin{proof}
\else
\begin{proof}[\ifIJM Proof\fi of \autoref{oddcasepreimage}]
\fi

We construct a preimage of the elementary matrix $X_{ij}$ as a map of the form
\[
\gamma_{ij}(g_k) =
\begin{cases}
w\cdot g_i \cdot g_j, & k=j \\
\phantom{w\cdot g_i \cdot {}} g_k, & k\ne j
\end{cases}
\]
with a suitable $w\in F_r$. By \autoref{autconstr}, this is an automorphism that stabilizes $U$ if the gene\-rator $g_j$ does not occur in $w$ and $\pi(w\cdot g_i) = \Id$.

To have $B(\gamma_{ij}) = X_{ij}$, the number of occurrences of all generators in $w$ must be zero each, that is $w\in [F_r,F_r]$. If the $i$-th loop actually is a looplet, we have $\pi(g_i)=\Id$ and we can choose $w=\Id$.

If $s_i>1$, this is where our preliminary calculations about permutations kick in. $g_k$ is another generator whose loop is not a looplet, i.e.\ $s_k>1$. If we only consider the set of cosets of $U$ that are on the loops $i$ or $k$,
\[
N_{ik} \da \{ U, g_i^1 U, \ldots, g_i^{s_i-1}U, \, g_k^1 U,\ldots, g_k^{s_k-1}U\},
\]
we can interpret the permutations $\pi(g_i)$ and $\pi(g_k)$ as elements of $\Sym(N_{ik})$. We are now in the situation of \autoref{altperm} with $\sigma = \pi(g_i)$ and $\omega = \pi(g_k)$, hence $A_{N_{ik}} \le \pi([\langle g_i,g_k\rangle ,\langle g_i,g_k\rangle])$. Since $\pi(g_i)$ is a cycle of odd length $s_i$, it is itself an even permutation, thus $\pi(g_i)\in A_{N_{ik}}$. This proves the existence of a word $w\in [\langle g_i,g_k\rangle ,\langle g_i,g_k\rangle]$ with $\pi(w) = \pi(g_i)^{-1}$, that is $\pi(w\cdot g_i) = \Id$. With this word, the automorphism $\gamma_{ij}$ is indeed in $\Stab_{\Aut(F_r)}(U)$ and a preimage of $X_{ij}$.
\end{proof}

\begin{lemma}
\label{squaredcasepreimage}
Let $U\le F_r$, $r\ge 3$, be an $s_1/\ldots/s_r$-loop subgroup and $i,j,\in\{1,\ldots,r\}$ with $i\ne j$. If $s_i = 1$ or there exists $k\in\{1,\ldots,r\}$, $k\ne i,j$, with $s_k>1$, then
\[
X_{ij}^2 \in B(\Stab_{\Aut(F_r)}(U)).
\]
\end{lemma}

\begin{proof}
Just as in the proof for \autoref{oddcasepreimage}, we will find preimages of squares of elementary matrices $X_{ij}^2$ as maps $\gamma_{ij}^{(2)}$ of the form 
\[
\gamma_{ij}^{(2)}(g_k) =
\begin{cases}
w\cdot g_i^2 \cdot g_j, & k=j \\
\phantom{w\cdot g_i^2 \cdot {}}g_k, & k\ne j.
\end{cases}
\]
Note that while $\pi(g_i)$ may not be an even permutation, $\pi(g_i^2) = \pi(g_i)^2 \in A_N$ certainly is. Therefore \autoref{altperm} provides us with a suitable $w\in F_r$ so that $\gamma_{ij}^{(2)}\in \Stab_{\Aut(F_r)}(U)$ by \autoref{autconstr}.
\end{proof}

\begin{lemma}
\label{evencasepreimage}
Let $U\le F_r$, $r\ge 3$, be an $s_1/\ldots/s_r$-loop subgroup and $i,j,k\in\{1,\ldots,r\}$ distinct with $s_i$ and $s_j$ even. Then 
\[
X_{ik}X_{jk} \in B(\Stab_{\Aut(F_r)}(U)).
\]
\end{lemma}

\begin{proof}
We again vary the construction of the previous two lemmata. Here we find automorphisms $\gamma_{ijk}$ of the form 
\[
\gamma_{ijk}(g_l) =
\begin{cases}
w\cdot g_i g_j \cdot g_k, & l=k \\
\phantom{w\cdot g_i g_j \cdot {}}g_l, & l\ne k.
\end{cases}
\]
such that $\gamma_{ijk}$ is a preimage of $X_{ik}X_{jk}$. We know that the cycles $\pi(g_i)$ and $\pi(g_j)$ are both of even length, thus odd permutations. Therefore, their pro\-duct $\pi(g_ig_j)\in A_N$ and with \autoref{altperm} we can find a suitable $w\in [\langle g_i,g_j\rangle, \langle g_i,g_j\rangle]$, such that $\gamma_{ijk} \in \Stab_{\Aut(F_r)}(U)$ by \autoref{autconstr}.
\end{proof}

\section{The lower bound}
\label{lowerboundsec}

We will first show $\Gamma_2 \le B(\Stab_{\Aut(F_r)}(U))$, i.e.\ $\Stab_{\Aut(F_r)}(U)$ is a congruence subgroup of level 2, and then consider the problem in $\GL_r(\Z/2\Z)$.

\begin{lemma}
\label{level}For a loop subgroup $U\le F_r$, $r\ge 3$, with at most $r-2$ looplets, we have
\[
\Gamma_2 \le B(\Stab_{\Aut(F_r)}(U)).
\]
\end{lemma}

\begin{proof}
We will find preimages in $\Stab_{\Aut(F_r)}(U)$ to the elements of a generating set of $\Gamma_2$. As preimages for the matrices $\{T_j, j=1,\ldots,r\}$ we choose the automorphisms $\tau_j$, which invert the $j$-th generator $g_j$ and do not modify the other generators.

If we have at most $r-3$ looplets, we find preimages to all squares of elementary matrices by \autoref{squaredcasepreimage}. By \autoref{gamma2gens}, these generate $\Gamma_2$.

If we have $r-2$ looplets, we fix a loop $i$ with $s_i=1$ and find preimages for each of
\[
\{X_{ik}, k\ne i\} \cup \{ X_{ji}^2, j\ne i\},
\]
invoking \autoref{oddcasepreimage} for the elementary matrices and \autoref{squaredcasepreimage} for the squares of elementary matrices. 
By \autoref{gamma2gensalt} these generate $\Gamma_2$.
\end{proof} 

Recall that $S(v)$, introduced in \autoref{stabsubgroups}, is the set of matrices stabilizing the row vector $v$ under multiplication from right.

\begin{proposition}
\label{mixedcase}For a loop subgroup $U\le F_r$, $r\ge 3$, with at most $r-2$ looplets, we have
\[
S(v) \le \overline{B(\Stab_{\Aut(F_r)}(U))}
\]
where
\[
v_i = 
\begin{cases}
0, & \text{if } s_i \text{ odd} \\
1, & \text{if } s_i \text{ even}.
\end{cases}
\]
\end{proposition}

\begin{proof}
By \autoref{stabgens}, $S(v)$ is generated by the elementary matrices $X_{ij}$ for $i\ne j$ and $v_i=0$ and the double elementary matrices $X_{ik}X_{jk}$ for $v_i=v_j=1$.

The double elementary matrices $X_{ik}X_{jk}$ are in $\overline{B(\Stab_{\Aut(F_r)}(U))}$ for $v_i=v_j=1$, as shown in \autoref{evencasepreimage}.

Let $i,j\in\{1,\ldots,r\}$, $i\ne j$ and $v_i=0$. We need to show that $X_{ij}$ is in $\overline{B(\Stab_{\Aut(F_r)}(U))}$. If $s_i=1$ or there is $k\in\{1,\ldots,r\}$ with $k\ne i,j$ and $s_k>1$, this follows from \autoref{oddcasepreimage}. If that is not the case, we are in the situation of $r-2$ looplets with $s_i>1$, $s_j>1$ and $s_k=1$ for $k\ne i,j$. Invoking \autoref{oddcasepreimage} for $X_{ik}$ and for $X_{kj}$, and using $[X_{ik},X_{kj}]=X_{ij}$, we obtain that $X_{ij}\in \overline{B(\Stab_{\Aut(F_r)}(U))}$.
\end{proof}

\section{The upper bound}
\label{upperboundsec}

To fully understand the situation, we yet have to find out whether the lower bound $S(v)$ from \autoref{mixedcase} is already the full group $\overline{B(\Stab_{\Aut(F_r)}(U))}$. It turns out that this is the case, and this group is the upper bound of the image of the stabilizer subgroup even for an arbitrary finite index subgroup $U$ of $F_r$:

\begin{proposition}
\label{generalupperbound}
Let $U\le F_r$ be a subgroup of finite index and let
\[
v_i \da 
\begin{cases}
0, & \text{if $\pi(g_i)$ is an even permutation} \\
1, & \text{if $\pi(g_i)$ is an odd permutation.} 
\end{cases}
\]
Then,
\[
\overline{B(\Stab_{\Aut(F_r)}(U))} \le S(v).
\]
\end{proposition}

\begin{proof}
Let $\gamma\in\Stab_{\Aut(F_r)}(U)$. For a generator $g_i$, $i\in\{1,\ldots,r\}$, the permutation $\pi(\gamma(g_i))$ has the same parity as $\pi(g_i)$, by \autoref{piprops} (3). Therefore, the number of generators in the word $\gamma(g_i)$ whose associated permutation is odd equals $v_i$ modulo 2:
\[
\sum_{j=1}^r \#_{g_j}\gamma(g_i) \cdot v_j \equiv v_i \pmod 2
\]
This equation can be written as $v\cdot \overline{B(\gamma)} = v$, hence $\overline{B(\gamma)} \in S(v)$. 
\end{proof}

This is an interesting result, as the subgroups $S(v)$ are maximal, but far from the only maximal groups.

\begin{theorem}
\label{sharpbound}
For a loop subgroup $U\le F_r$, $r\ge 3$, with at most $r-2$ looplets, $B(\Stab_{\Aut(F_r)}(U))$ is a congruence subgroup of level 2 and its image in $\GL_r(\Z/2\Z)$ is $S(v)$ where
\[
v_i = 
\begin{cases}
0, & \text{if } s_i \text{ odd} \\
1, & \text{if } s_i \text{ even}.
\end{cases}
\]
\end{theorem}

\begin{proof}
\autoref{mixedcase} applies because $\Gamma_2\le\overline{B(\Stab_{\Aut(F_r)}(U))}$ by \autoref{level} and the permutation $\pi(g_i)$ is odd if and only if $s_i$ is even. \autoref{generalupperbound} provides the inclusion in the other direction.
\end{proof}

%
%

\section{The excluded case}
\label{excludedcasesec}

In the previous sections, we have always excluded loop subgroups with exactly $r-1$ looplets. These subgroups have some special properties that make them break rank and therefore, they are handled separately here.

Let $U$ be a loop subgroup with exactly $r-1$ looplets, and, without loss of generality, assume that $s_1\ne 1$. The subgroup can now be written as 
\[
U = \{ w\in F_r \mid \#_{g_1} w \equiv 0 \pmod {s_1}\}.
\]

Let $A:F_r \to \Z^r$ be the surjective abelianization map defined by $w \mapsto (\#_{g_i} w)_{i=1,\ldots,r}$. Then $U = A^{-1}(U')$ with $U' \da \{ v\in\Z^r \mid v_1 \equiv 0 \mod s_1\}\le \Z^r$. This allows us to calculate $B(\Stab_{\Aut(F_r)}(U))$ using the following, more general observation:

\begin{remark}
\label{stabofpreimage}
Let $\varphi : G \to G'$ be a surjective group homomorphism with its kernel characteristic in $G$. Let $H\le G$ and $H'\le G'$ be subgroups such that $H = \varphi^{-1}(H')$. Let $\psi:\Aut(G)\to \Aut(G')$ be the map induced by $\varphi$. Then
\[
\psi(\Stab_{\Aut(G)}(H)) = \Stab_{\Aut(G')}(H') \cap \psi(\Aut(G)).
\]
\end{remark}

\ifIJM
\else
\begin{proof}
For $\gamma\in\Aut(G)$, $y\in G'$ the map $\psi:\Aut(G) \to \Aut(G')$ is defined by $\psi(\gamma)(y) \da \varphi(\gamma(x))$ for an $x\in\varphi^{-1}(y)$. This is well-defined because the kernel of $\varphi$ is characteristic. Thus
\begin{align*}
\psi(\Stab_{\Aut(G)}(H)) 
&= \psi(\{\gamma \in \Aut(G) \mid \gamma(H) = H\}) \\
&= \{\psi(\gamma) \mid \gamma \in \Aut(G),\ \forall x\in H: \gamma(x) \in H\} \\
&= \{\psi(\gamma) \mid \gamma \in \Aut(G),\ \forall x\in H: \varphi(\gamma(x)) \in \varphi(H)\} \\
&= \{\psi(\gamma) \mid \gamma \in \Aut(G),\ \forall y\in H': \varphi(\gamma(\varphi^{-1}(y))) \in \varphi(H)\} \\
&= \{\psi(\gamma) \mid \gamma \in \Aut(G),\ \forall y\in H': \psi(\gamma)(y) \in H'\} \\
&= \{\gamma' \in \Aut(G') \mid \forall y\in H': \gamma'(y) \in H'\} \cap \psi(\Aut(G))\\
&= \Stab_{\Aut(G')}(H') \cap \psi(\Aut(G)) \ifIJM\qedhere\fi
\end{align*}
\end{proof}
\fi

\begin{proposition}
\label{excludedcase}
For the $s_1/1/\ldots/1$ loop subgroup $U$, $s_i\in \N$, of $F_r$ we have
\[
\Gamma_{s_1} \le B(\Stab_{\Aut(F_r)}(U)) = \Stab_{\GL_r(\Z)}(U') \le \GL_r(\Z)
\]
where
\[
U' \da \{v \in \Z^r \mid v_1 \equiv 0\mod s_1\}.
\]
\end{proposition}

\begin{proof}
As noted before, $U = A^{-1}(U')$. $B$ is surjective, because preimages of a generating set of $\GL_r(\Z)$ are given by $\tau_1$ and the automorphisms in the proof of \autoref{oddcasepreimage}. Therefore, \autoref{stabofpreimage} gives us
\[
B(\Stab_{\Aut(F_r)}(U)) = \Stab_{\GL_r(\Z)}(U'). 
\]

Since every matrix in $\Gamma_{s_1}$ stabilizes $U'$, $B(\Stab_{\Aut(F_r)}(U))$ contains $\Gamma_{s_1}$.
\end{proof}

As a matter of fact, $s_1$ is the level of the congruence subgroup $B(\Stab_{\Aut(F_r)}(U))$. The difference to the other loop subgroups is evident: While here we easily reach any congruence subgroup level, in the other cases we do not exceed level 2.

\bibliography{bib}
\ifIJM
\bibliographystyle{ijmart}
\else
\bibliographystyle{amsalpha}
\fi
\end{document}